\documentclass[10pt]{amsart}

\usepackage{amsmath,amsfonts, latexsym,graphicx, footmisc, amssymb, mathtools, enumerate}

\usepackage{hyperref}
\hypersetup{colorlinks=true, urlcolor=blue, citecolor=blue, linkcolor=blue}

\newcommand{\U}{{\mathcal U}}
\newcommand{\0}{{\mathbf 0}}
\newcommand{\C}{{\mathbb C}}
\newcommand{\Z}{{\mathbb Z}}

\newcommand{\Proj}{{\mathbb P}}

\newtheorem{defn0}{Definition}[section]
\newtheorem{prop0}[defn0]{Proposition}
\newtheorem{conj0}[defn0]{Conjecture}
\newtheorem{thm0}[defn0]{Theorem}
\newtheorem{lem0}[defn0]{Lemma}
\newtheorem{corollary0}[defn0]{Corollary}
\newtheorem{example0}[defn0]{Example}
\newtheorem{remark0}[defn0]{Remark}
\newtheorem{question0}[defn0]{Question}
\newtheorem{exercise0}[defn0]{Exercise}

\newenvironment{thm}{\begin{thm0}}{\end{thm0}}
\newenvironment{lem}{\begin{lem0}}{\end{lem0}}
\newenvironment{cor}{\begin{corollary0}}{\end{corollary0}}

\newenvironment{exm}{\begin{example0}\rm}{\end{example0}}

\newcommand{\thmref}[1]{Theorem~\ref{#1}}

\newcommand{\corref}[1]{Corollary~\ref{#1}}

\newcommand{\mbf}[1]{{\mathbf #1}}

\title{A Lemma of Lazarsfeld and the Jacobian blow up}



\author{David B. Massey}

\date{}

\begin{document}

\begin{abstract} For a complex analytic function $f$, the exceptional divisor of the jacobian blow-up is of great importance. In this paper, we show what a lemma from the thesis of Lazarsfeld tells one about the structure of this exceptional divisor.
\end{abstract}

\maketitle

%\newpage

%\tableofcontents

%\newpage

\section{introduction} Let $\U$ be an open subset of $\C^{n+1}$ and let $f:(\U, \0)\rightarrow (\C, 0)$ be a nowhere locally constant complex analytic function. Near the origin, the critical locus $\Sigma f$ of $f$ is contained in the hypersurface $V(f)$ defined by $f$; we assume that $\U$ is chosen small enough so that this is true throughout $\U$. We use $\mbf z:=(z_0, \dots, z_n)$ for the coordinates on $\C^{n+1}$ and so on $\U$.

We let $\pi:\operatorname{Bl}_{j(f)}\U\rightarrow \U$ be the projection map of the blow-up of $\U$ along the jacobian ideal $$j(f):=\left\langle\frac{\partial f}{\partial z_0}, \dots, \frac{\partial f}{\partial z_n}\right\rangle,$$
where $\operatorname{Bl}_{j(f)}\U\subseteq \U\times \Proj^n$. Let $E=\pi^{-1}(\Sigma f)$ denote the exceptional divisor, which is purely $n$-dimensional. Let $W_0$, $W_1$, \dots, $W_r$ denote the distinct  irreducible components of $E$ over $\0$, i.e., the irreducible components $W$ of $E$ such that $\0\in\pi(W)$. 

Now let us identify the $\U\times \Proj^n$ which contains the jacobian blow-up with the projectivized cotangent space $\Proj(T^*\U)$. Under this identification, $\operatorname{Bl}_{j(f)}\U$ is equal to the projectivized closure of the relative conormal of $f$, that is, 
$$\operatorname{Bl}_{j(f)}\U=\Proj\big(\overline{T^*_{f_{|\U\backslash\Sigma f}}\U}\big).$$
For $0\leq k\leq r$, we let $Y_k$ denote the irreducible analytic set $\pi(W_k)$.  Then the fact that $E$ is purely $n$-dimensional, combined with the existence of an $a_f$ stratification, tells us that $W_k$ is equal to the closure of the conormal space of the regular part $Y_k^\circ$ of $Y_k$, that is,  $W_k=\Proj\big(\overline{T^*_{Y_k^\circ}\,\U}\big)$. We refer to the $Y_k$ as {\bf the Thom varieties of $f$ at $\0$}. 

\bigskip

Clearly, the Thom varieties are important for understanding limiting relative conormals and the $a_f$ condition. In addition, the result of \cite{kashnotes} and \cite{lemeb} tells us that, as sets, the exceptional divisor $E$ is equal to the projectivized characteristic cycle of the sheaf of vanishing cycles of $\C^\bullet_{\U}$ (or $\Z^\bullet_{\U}$) along $f$. Thus, the Thom varieties are also closely related to the topology of the Milnor fibers of $f$ at points in $\Sigma f$.

\bigskip

While we prove a more general result in \thmref{thm:thetheorem}, a special case is much easier to state:

\medskip

\noindent {\bf Corollary}. {\it Suppose that $\Sigma f$ is smooth at $\0$. Then, at $\0$, either

\begin{enumerate} 

\item there is a Thom variety of $f$ of codimension 1 in $\Sigma f$, or

\smallskip

\item $\Sigma f$ itself is the only Thom variety, and $f$  defines a family of isolated singularities with constant Milnor number (that is, a simple $\mu$-constant family as given in Definition 1.1 of \cite{lemassey}).
 
 \end{enumerate}
 }
 
 So, if $\Sigma f$ is smooth and there are any proper sub-Thom varieties in $\Sigma f$, then there must be one of codimension $1$. We find this somewhat surprising.
 
 \smallskip
 
 The crux of the proof of our theorem lies in a lemma from the Ph.D. thesis \cite{lazarsfeld} of Lazarsfeld, which we recall in the next section.

\section{The Lemma and the Theorem}

We now state Lemma 2.3 of \cite{lazarsfeld} (with some changes in notation):

\begin{lem}\label{lem:laz}\textnormal{(Lazarsfeld)} Let $Z$ be an irreducible normal variety of dimension $n+1$, and let $X\subseteq Z$ be a subvariety which is locally defined (set-theoretically)  by $n+1-e$ equations. Fix an irreducible component $V$ of $X$. Then, for all $x\in V\cap\overline{X\backslash V}$, $\dim_x \big(V\cap\overline{X\backslash V}\big)\geq e-1$.
\end{lem}

\medskip

The proof of our theorem below uses the above lemma in a crucial way. Our proof also uses, as sets, our notation and results on relative polar cycles and L\^e cycles as presented in \cite{lecycles}. We also use the characterization of L\^e cycles given by Remark 10.7 and Corollary 10.5 of \cite{lecycles}, which tells us that, as a set, the L\^e cycle of dimension $j$ is the union of the $j$-dimensional absolute polar cycles of the Thom varieties of dimension $\geq j$ (here, we use the result of \cite{kashnotes} and \cite{lemeb} which tells us that, as sets, the exceptional divisor $E$ of the Jacobian blow-up is equal to the projectivized characteristic cycle of the sheaf of vanishing cycles of $\C^\bullet_{\U}$ or $\Z^\bullet_{\U}$ along $f$.) The reader may see Section 7 (and especially Theorem 7.5) of \cite{numinvar} for the details.

We should mention that our theorem and its proof are closely related to Proposition 1.31 from \cite{lecycles}. However, the hypotheses of that proposition are confusing, the result depends on a coordinate choice, the conclusion is not as general as it could be, and the proof omits important details.

\medskip

\begin{thm}\label{thm:thetheorem} \textnormal{(Thom Going Down) } Let $V$ be an irreducible smooth component of $\Sigma f$ at $\0$. Fix $e\leq \dim V$. Suppose that, at $\0$, for all $j\geq e$, the Thom varieties of $f$ of dimension $j$, which are contained in $V$, are smooth. Then, at $\0$, one of the following must hold:

\begin{enumerate} 

\item $\dim_\0V\cap\overline{\Sigma f\backslash V}\geq e-1$, that is, $V$ intersected with the union of the other irreducible components of $\Sigma f$ has dimension at least $e-1$, or

\smallskip

\item there is a Thom variety of $f$ of dimension $e-1$ inside $V$, or

\smallskip

\item $V=\Sigma f$, and there are no Thom varieties of $f$ of dimension less than $e$.
 
 \end{enumerate}
\end{thm}

\begin{proof} Throughout, we work at the origin, i.e., consider the germ of the situation at the origin. 

\medskip

Let 
$$
X:=V\left(\frac{\partial f}{\partial z_e}, \dots, \frac{\partial f}{\partial z_n}\right).
$$
Note that all components of $X$ must have dimension $\geq e$.

Let $\Gamma^e_{f, \mbf z}$ denote the union of the components of $X$ which are {\bf not} contained in $\Sigma f$. By Theorem 1.28 of \cite{lecycles}, we may make a generic linear change of  coordinates (which we still write as $\mbf z$)  so that $\Gamma^e_{f, \mbf z}$ is purely $e$-dimensional (which vacuously allows for $\Gamma^e_{f, \mbf z}$ to be empty); $\Gamma^e_{f, \mbf z}$ is the relative polar variety of dimension $e$. The same theorem implies that we may also select our coordinates so that all of the L\^e cycles have proper dimension.

By our smoothness assumption on the Thom varieties, we may also assume that our coordinates are chosen so that, for all Thom varieties $W\subseteq V$ such that $\dim W\geq e$, the restriction of the maps $(z_0, \dots, z_m)$ to $W$ have no critical points for $m<\dim W$, i.e., $V(z_0, \dots, z_m)$ transversely intersects $W$ for $m<\dim W$. In terms of absolute polar varieties (for our notation and definition, see Definition 7.1 of \cite{numinvar}), this means that, for all Thom varieties $W\subseteq V$ such that $e\leq\dim W$,  the absolute polar varieties (as sets) $\Gamma^m_{\mbf z}(W)$ are empty for $m<\dim W$. Note that $\Gamma^{\dim W}_{\mbf z}(W)=W$.

\smallskip

It is trivial that $X=\Gamma^e_{f, \mbf z}\cup \Sigma f$. Note also that an irreducible component of $\Sigma f$ of dimension $<e$ can not be an irreducible component of $X$ (its dimension is too small), but rather must be contained in an irreducible component of $\Gamma^e_{f, \mbf z}$. However, since our coordinates are such that $\Gamma^e_{f, \mbf z}$ is purely $e$-dimensional, every component of $\Sigma f$ of dimension $\geq e$ must be an irreducible component of $X$. In particular, $V$ is an irreducible component of $X$.

\smallskip

We now apply Lazarsfeld's Lemma. It tells us that, if $\0\in V\cap\big(\Gamma^e_{f, \mbf z}\cup \overline{\Sigma f\backslash V}\big)$, then either
$$
\dim_{\0} V\cap \Gamma^e_{f, \mbf z}\geq e-1 \hskip 0.2in \textnormal{ or }\hskip 0.2in \dim_{\0} V\cap  \overline{\Sigma f\backslash V}\geq e-1.
$$

\medskip

Suppose that we are not in case (1) of the theorem, i.e., suppose that $\dim_\0V\cap\overline{\Sigma f\backslash V}<e-1$. Then either $\dim_{\0} V\cap \Gamma^e_{f, \mbf z}\geq e-1$, or $\0\not\in \Gamma^e_{f, \mbf z}$ and $\0\not \in\overline{\Sigma f\backslash V}$. We claim that these correspond to cases (2) and (3), respectively.

\bigskip

\noindent Case 2:

\smallskip

Suppose that $\dim_\0V\cap\overline{\Sigma f\backslash V}<e-1$ and $\dim_{\0} V\cap \Gamma^e_{f, \mbf z}\geq e-1$.

\smallskip

By Proposition 1.15 of \cite{lecycles}, as sets, 
$$
V\cap \Gamma^e_{f, \mbf z}=V\cap\bigcup_{j\leq e-1}\Lambda^j_{f, \mbf z},
$$
where $\Lambda^j_{f, \mbf z}$ is the purely $j$-dimensional L\^e cycle. Thus, $\dim_{\0} V\cap \Gamma^e_{f, \mbf z}\geq e-1$ implies that $V$ contains an $(e-1)$-dimensional irreducible component $Y$ of $\Lambda^{e-1}_{f, \mbf z}$ at the origin. By Corollary 10.15 and/or Theorem 10.18 of \cite{lecycles}, $Y$ is a component of an absolute polar variety of a Thom variety $T$ (which necessarily must have dimension $\geq e-1$) of $f$; however, since we are assuming that $\dim_\0V\cap\overline{\Sigma f\backslash V}<e-1$, $T$ cannot be contained in another irreducible component of $\Sigma f$, but rather must be contained in $V$. But by our hypotheses, the Thom varieties in $V$ of dimension $>e-1$ are smooth and have no absolute polar varieties of dimension $(e-1)$. Thus $T$ must be $(e-1)$-dimensional and so $Y=T$, and we have the conclusion of Case 2.

\bigskip

\noindent Case 3:

\smallskip

Now suppose that $\0\not\in \Gamma^e_{f, \mbf z}$ and $\0\not \in\overline{\Sigma f\backslash V}$. First, $\0\not \in\overline{\Sigma f\backslash V}$ immediately implies that $V=\Sigma f$ at the origin. As we saw in Case 2, but using that $V=\Sigma f$, we have
$$
V\cap \Gamma^e_{f, \mbf z}=\bigcup_{j\leq e-1}\Lambda^j_{f, \mbf z},
$$
and, as we are assuming that $\0\not\in \Gamma^e_{f, \mbf z}$, this implies that, at the origin, $\Lambda^j_{f, \mbf z}=\emptyset$ for all $j\leq e-1$. But $\Lambda^j_{f, \mbf z}$ includes any Thom variety of dimension $j$, and so there are none for $j\leq e-1$. Therefore, we have the conclusion of Case 3.
\end{proof}

\medskip

Letting $e=d$ in the theorem above, we obtain the corollary from the introduction:

\medskip

\begin{cor}\label{cor:cor} Suppose that $\Sigma f$ is smooth at $\0$. Then, at $\0$, either

\begin{enumerate} 

\item there is a Thom variety of $f$ of codimension 1 in $\Sigma f$, or

\smallskip

\item $\Sigma f$ itself is the only Thom variety, and $f$ defines a family of isolated singularities with constant Milnor number, that is, a simple $\mu$-constant family as given in Definition 1.1 of \cite{lemassey}.
 
 \end{enumerate}

\end{cor}
\begin{proof} If we let $e=d$ in \thmref{thm:thetheorem}, we obtain essentially the whole corollary. If $\Sigma f$ is smooth, then it is irreducible, and Case 1 from the theorem cannot occur. Cases 2 and 3 from the theorem correspond to Cases 1 and 2, respectively, of the corollary. The only thing that requires further proof is that Case 3 of the theorem implies that $f$ defines a family of isolated singularities with constant Milnor number.

However, Case 3 of \thmref{thm:thetheorem} is the case where $\0\not\in \Gamma^d_{f, \mbf z}$, that is, $\Gamma^e_{f, \mbf z}$ is empty near the origin or, with its cycle structure, is $0$. Then one applies the equivalence of Conditions 3 and 5 from Theorem 2.3 of \cite{lemassey} to conclude that $f$ defines a family of isolated singularities with constant Milnor number (a simple $\mu$-constant family).
\end{proof}

\medskip

We can use \thmref{thm:thetheorem} to prove a version of itself which refers to super-Thom varieties rather than sub-Thom varieties.

\medskip

\begin{thm}\label{thm:thetheorem2} \textnormal{(Thom Going Up) } Let $T$ be an $r$-dimensional Thom variety of $\Sigma f$ at $\0$. Let $V\supseteq T$ be an irreducible component of $\Sigma f$ at $\0$.  Then, at $\0$, one of the following must hold:

\begin{enumerate} 

\item $T=V$, or

\smallskip

\item $T\subseteq V\cap\overline{\Sigma f\backslash V}$, or

\smallskip

\item there exists a Thom variety $T'\subseteq V$ of $f$ at $\0$ such that $T\subseteq\Sigma T'$, or

\smallskip

\item there exists a Thom variety $T'\subseteq V$ of $f$ at $\0$ such $\dim T'=r+1$ and $T\subseteq T'$.
 
 \end{enumerate}
\end{thm}

\begin{proof} Suppose that we are not in Cases 1, 2, or 3. Then, let
$$
X:=T\backslash \Big(\overline{\Sigma f\backslash V} \ \cup \ \bigcup_{T'}\Sigma T' \ \cup \ \bigcup_{T''\not\supseteq T} T''\Big),
$$
where the unions are over all Thom varieties  $T'$ and $T''$ contained in $V$, and $T''\not\supseteq T$. Since we are not in Cases 2 or 3, $X$ is an open, dense subset of $T$. Let $x\in X$. Then, $x\not\in\overline{\Sigma f\backslash V}$ and, at $x$, every Thom variety $T'\subseteq V$ contains $T$ and is smooth at $x$.

\smallskip

We apply \thmref{thm:thetheorem} at $x$ in place of $\0$. Let $e$ be the smallest dimension of a Thom variety $T'\subseteq V$ at $x$ such that $T'$ properly contains $T$; there is such an $e$ since we are not in Case 1, i.e., $V$ itself is a Thom variety in $V$ which properly contains $T$. Then we must be in Case 2 of \thmref{thm:thetheorem}, and there must be a Thom variety $\widetilde T$ of dimension $(e-1)$ in $V$ at $x$. But this $\widetilde T$ must contain $T$ (by the choice of $x$), and we would have a contradiction of the definition of $e$ unless $\widetilde T$ does not {\bf properly} contain $T$. Thus we must have $\widetilde T=T$ at $x$ and $e=r+1$. Since this is true for $x$ in an open, dense subset of $T$, the conclusion of Case 4 follows.
 \end{proof}

\section{Examples}

\begin{exm} Suppose that the irreducible components of $\Sigma f$ at the origin are a line $L$ and a plane $P$. Is it possible that $L$ and $P$ are the only Thom varieties of $f$ at $\0$? The answer is ``no'', and one might suspect that that is because $\{\0\}$ must also be a Thom variety. However, \thmref{thm:thetheorem} with $e=2$ tells us  that, in fact, the plane $P$ must contain a $1$-dimensional Thom variety.

Let us look at a specific example. Let $f=w^2+xyz^2$. Then,
$$
\Sigma f=V\big(2w, yz^2, xz^2, 2xyz\big)=V(w, z)\cup V(w, y, x) = P\cup L.
$$
Of course, $P$ and $L$ are Thom varieties, but  \thmref{thm:thetheorem}  tells us that there must be a $1$-dimensional Thom variety contained in $P$.

The reader is invited to calculate the blow-up of the jacobian ideal to show that the other Thom varieties are, in fact, $V(w, z, x)$, $V(w, z, y)$ and $\{\0\}$.
\end{exm}

\medskip

\begin{exm} Is the smoothness requirement in \thmref{thm:thetheorem} and \corref{cor:cor} really necessary? Yes. Consider
$f=w^2+(x^2+y^2+z^2)^2$. Then,
$$
\Sigma f=V\big(2w, \,4(x^2+y^2+z^2)x,\,  4(x^2+y^2+z^2)y,\,  4(x^2+y^2+z^2)z\big)=V(w, x^2+y^2+z^2).
$$
Now $V(w, x^2+y^2+z^2)$ is a Thom variety. However, by symmetry, there cannot be a $1$-dimensional Thom variety inside $V(w, x^2+y^2+z^2)$ and yet, by direct calculation, one can show that $\{\0\}$ is  a Thom variety. 

Thus, if $\Sigma f$ is irreducible, but not smooth, the conclusions of  \corref{cor:cor}  need not hold.
\end{exm}

\bibliographystyle{plain}

\bibliography{Masseybib}

\end{document}